\documentclass{cmslatex}
\usepackage[paperwidth=7in, paperheight=10in, margin=.875in]{geometry}
 \usepackage[backref,colorlinks,linkcolor=red,anchorcolor=green,citecolor=blue]{hyperref}
\usepackage{amsfonts,amssymb}
\usepackage{amsmath}
\usepackage{graphicx}
\usepackage{cite}
\usepackage{enumerate}

\renewcommand{\theequation}{\arabic{section}.\arabic{equation}}

   \allowdisplaybreaks
\begin{document}
\title{Stability of Large Amplitude Viscous Shock Wave for 1-D Isentropic
Navier-Stokes System in the Half Space \thanks{Received date, and accepted date (The correct dates will be entered by the editor).}}
\author{Lin,Chang\thanks{School of Mathematics   Science, Beihang University, Beijing, China, ( changlin23@buaa.edu.cn). }}
\pagestyle{myheadings} \markboth{Large amplitude shock wave for impermeable wall problem}{Lin} \maketitle
\begin{abstract}
In this paper, the asymptotic-time behavior of solutions to an initial boundary value problem in the half space for 1-D isentropic Navier-Stokes system is investigated. It is shown that the viscous shock wave is stable for an impermeable wall problem where the velocity is zero on the boundary provided that the shock wave is initially far away from the boundary. Moreover, the strength of shock wave could be arbitrarily large. This work essentially improves the result of [A. Matsumura, M. Mei, Convergence to travelling fronts of solutions of the p-system with viscosity in the presence of a boundary, Arch. Ration. Mech. Anal., 146(1): 1-22, 1999], where the strength of shock wave is sufficiently small.
\end{abstract}


\section{Introduction}\label{intro}

We consider a 1-D  isentropic Navier-Stokes system  for general   viscous gas, which reads in the Lagrangian coordinate  as,
\begin{equation}\label{1.1}
\left\{ \begin{array}{ll}
&v_t-u_x=0,   \\
&u_t+p_x=(\mu(v)\frac{u_x}{v^{ }})_x,
\end{array} \right.
\end{equation}

where $t>0, x\in \mathbb{R_+},$  and $v(x,t)=\frac{1}{\rho(x,t)}$ is the specific volume, $u(x,t)$   the fluid velocity, $p=a v^{-\gamma}$  the pressure with constant $a>0$, $\gamma> 1$  the adiabatic constant, and $\mu(v)=\mu_{0}v^{-\alpha}$   the viscosity coefficient with $\alpha\geq 0$. When the viscosity $\mu(v)\equiv 0 $, the system (\ref{1.1}) becomes  the famous Euler system
\begin{equation}
\left\{ \begin{array}{ll}\label{1.2}
&v_t-u_x=0,\\
&u_t+p_x=0,
\end{array} \right.
\end{equation}
that has rich wave phenomena such as shock and rarefaction waves. When $\mu (v)>0 $, the shock wave is mollified as the so-called viscous shock wave. Without loss of generality, we assume $\mu_0=1$ in what follows.

 Since the system \eqref{1.1} is regular than the Euler one \eqref{1.2}, it is very interesting and important to study the stability of the viscous version of shock wave, i.e., the viscous shock wave, for the viscous conservation laws such as the NS system \eqref{1.1} with the initial data:
\begin{equation}\label{1.3}
(v, u)(x, 0) = (v_0,u_0) (x)\longrightarrow (v_{\pm},u_{\pm}),\quad \text{as} \quad  x\rightarrow \infty.
\end{equation}

The stability of viscous shock wave for the Cauchy problem (\ref{1.1}), (\ref{1.3}) has been extensively studied in a large amount of literature since the pioneer works of \cite{g1986,mn1985}, see the other interesting works \cite{fs1998, hm2009, hlz2017,km1985,l1997,lz2009,lz2015,m,mn1994,sx1993}. It is noted that most of above works require the strength of shock wave is suitably small, that is, the shock is weak. The stability of large amplitude shock (strong shock) is more interesting and challenging in both mathematics and physics, see  works \cite{hh2020, mn1985 , km1985 ,mz2004,mw2010,vy2016,z2004}.

Matsumura-Nishihara \cite{mn1985} showed that the viscous shock wave is stable if $ |v_+-v_-|<C(\gamma - 1)^{-1}$, that is, when $ \gamma \rightarrow 1  $, the strength of shock wave could be large. This condition is later relaxed in \cite{km1985} to the condition that $ |v_+-v_-|<C(\gamma - 1)^{-2}$. Recently, the restriction on the strength of shock  was removed in  \cite{mw2010}   by an elegant weighted energy method as $\alpha >\frac{\gamma-1} {2} $. Vasseur-Yao  \cite{vy2016} removed the condition $\alpha >\frac{\gamma-1} {2}$ by introducing a beautiful variable transformation.  Moreover,  He-Huang  \cite{hh2020} extended the result of \cite{vy2016} to general pressure $p(v)$ and general viscosity $\mu(v)$, where $\mu(v)$ could be any positive smooth function.

On the other hand, it is also interesting to investigate the stability of viscous shock wave under the effect of boundary. In 1999, Matsumura-Mei \cite{mm1999} considered an impermeable wall problem of \eqref{1.1} in the half space $x\ge 0$, i.e.,
\begin{equation}\label{1.4}
\left\{ \begin{array}{ll}
(v,u)(x,0)=(v_0,u_0)(x) \longrightarrow (v_+,u_+), ~x\to +\infty,\\
u(0,t)=0, ~~ t\in\mathbb{R_+},
\end{array} \right.
\end{equation}
where $v_+>0, u_+<0$. The impermeable wall means that there is no flow across the boundary so that  the velocity at the boundary $x=0$ has to be zero. It was proved in \cite{mm1999} that the solution of \eqref{1.1}, \eqref{1.4} with $\alpha=0$ time-asymptotically tends to an outgoing shock wave (2-shock) connecting the left state $(v_-,0)$ and the right one $(v_+,u_+)$
if  $ |v_+-v_-|<C(\gamma - 1)^{-2}$,  and  the outgoing shock is initially  far away from the boundary so that the interaction between the shock and the boundary is weak,
\color{black}{}
where $v_-$ is determined by the RH condition, i.e.,
\begin{eqnarray}\label{1.5}
\left\{\begin{array}{ll}
-s(v_+-v_-)-(u_+-u_-)=0, \\
-s(u_+-u_-)+(p(v_+)-p(v_-))=0,
\end{array}
\right.
\end{eqnarray}
with $u_-=0$.
Matsumura-Nishihara \cite{mn2004}  removed the condition that the shock is initially far away from the boundary by extending the half space to the whole space, with the price that  the shock wave has to be weak even for $\gamma=1$ case.

In this paper, we aim to prove that the large amplitude shock wave is still stable for the impermeable wall problem \eqref{1.1}-\eqref{1.4}. Roughly speaking,  there exists a 2-viscous shock wave (outgoing shock) $(V_2,U_2)$ connecting $(v_-, 0  )$ and $ (v_+, u_+   )$ with $ v_- $  determined by the RH condition (\ref{1.5}), and $(V_2,U_2)$ is asymptotically stable if it is initially far away from the boundary. The precise statement of the main result is given in Theorem \ref{theorem}.

We outline the strategy as follows. Motivated by \cite{vy2016} and \cite{hh2020}, we introduce a new variable $h=u -v^{(-\alpha+1)} v_{x}$  and formulate a new equation $\eqref{4.2}_2$ in which the viscous term is moved to the mass equation $\eqref{4.2}_1$ so that the two nonlinear terms ${}{p}_x$ and $(\frac{{}{v}_x}{{}{v}^{\alpha+1}})_x $ are decoupled  and the interaction between nonlinear terms is weaken. Since the strength of outgoing shock   is arbitrarily large, the interaction  between the  2-shock and the boundary $x=0$  is strong. We have to assume that the outgoing shock is initially far away from the boundary so that the interaction is weak.  Since the  boundary terms with first order derivatives are controlled, we can obtain the low order  estimates through careful analysis.
But  the idea using the new system \eqref{4.2} does not work in the higher order estimation since it is very difficult to control the  second order derivatives of boundary terms for the new variable $h$. Note that the second derivatives of $u$ on the boundary can be controlled, we then turn to original system \eqref{1.1} to obtain the higher order energy estimates, and finally complete the a priori estimates.
\

The rest of the paper will be arranged as follows. In section \ref{Sec.2}, the outgoing shock wave is formulated and the main result is stated.   In section \ref{Sec.3}, the problem is reformulated by the anti-derivatives of the perturbations around the viscous shock wave. In section \ref{Sec.4}, the a priori estimates are established. In section \ref{Sec.5}, the main theorem is proved.

\
\textbf { Notation.}
The functional $\|\cdot\|_{L^p(\Omega)}$ is defined by $\| f\|_{L^p(\Omega)} =  (\int_{\Omega}|f|^{p}(\xi )\operatorname{ d }\xi )^{\frac{1}{p}}$. The symbol $\Omega$ is often omitted, when $\Omega=(0,\infty)$. As $p=2$, for simplicity we denote,
\begin{equation*}
\| f\| =  \left(\int_{ 0}^{ \infty}f^{2}(\xi )\operatorname{ d }\xi \right)^{\frac{1}{2}}.
\end{equation*}
\color{black}{}
In addition, $H^m$ denotes the  $m$-th  order Sobolev space of functions defined by
\begin{equation*}
\|f\|_{m} =  \left( \sum_{k=0}^{m}  \|\partial^{k}_{\xi}f\|^2 \right)^{\frac{1}{2}}.
\end{equation*}

\section{Preliminaries and Main Theorem}\label{Sec.2}
\subsection{ Viscous Shock Profile and Location of the Shift.}As pointed out by \cite{mm1999},  the solution of the impermeable wall problem \eqref{1.1}-\eqref{1.4} is expected to tend toward the outgoing viscous shock $(V, U)(\xi)$ satisfying
\begin{equation}
\left\{ \begin{array}{ll}
&{-s }{V}'-{U}'=0,\\
&{-s }U'+p(V)'=\left(\frac{U'}{V^{\alpha+1}}\right)',\\
&(V ,U)(-\infty)=(v_-,0),\quad (V ,U)(+\infty)=(v_+,u_+),
\end{array} \right.
\label{2.1}\end{equation}
where $'=\operatorname{d} / \operatorname{d\xi}$, $\xi=x-s t$, $s $ is the shock speed determined by the RH condition \eqref{1.5} and $v_\pm>0,u_+<0$ are given constants. From $(\ref{2.1})_1 $ and $(\ref{2.1})_2 $, one gets
\begin{align}
\begin{split}
&  s^{2}V'+p(V)'= -\left(\frac{s V'}{V^{\alpha+1}}\right)'.
\end{split}&
\label{2.2}\end{align}
Integrating (\ref{2.2}) over $(\pm\infty,\xi)$ gives
\begin{align}\label{2.3}
\begin{split}
&\frac{s  V'}{V^{\alpha+1}}=-s^{2}V -p(V) -b=:h(V), V(\pm\infty)=v_\pm,
\end{split}&
\end{align}
\begin{align}
\begin{split}
&U=-s (V-v_-)=-s (V-v_+)+u_+,
\end{split}&
\end{align}
where $b=-s ^2 v_\pm-p(v_\pm) $.

\begin{proposition} [\cite{mm1999}]
There exists a unique viscous shock profile $(V , U)(\xi)$ up to a shift  satisfying
\begin{eqnarray}\label{2.5}
	0<v_{-}<V(\xi)<v_{+}, \quad
	h(V )>0, \quad U'<0,
\end{eqnarray}
\begin{eqnarray}
\left|V (\xi)-v_{\pm}\right|=O(1)\left|v_{+}-v_{-}\right|  e^{-C_{\pm}|\xi|},
\end{eqnarray}
as $\xi \rightarrow \pm \infty,$ where $C_\pm=\frac{v_{\pm}^{\alpha+1}}{   s_{ }} |p'(v_\pm)+s_{} ^2|$,
$s_{}=\frac{-u_+}{v_+-v_-}. $
\end{proposition}

We expect  $\int_{0}^{\infty} [v(x,t)-V(x-st+\beta_{0}-\beta)] \operatorname{d}x\rightarrow 0$ as $t\rightarrow\infty $. As in \cite{mm1999}, the shift of  viscous shock profile is given by
\begin{eqnarray}
\beta_{0}=\frac{1}{v_+-v_-} \left\{\int_{0}^{\infty}[{v}_{0}(x)-V (x-\beta)]\operatorname{d}x+\int_{0}^{\infty}U(-st-\beta)  \operatorname{d}t\right\}.
\label{2.7}\end{eqnarray}

\subsection{Main Theorem.}We assume that for $ \beta > 0$, the initial data satisfies
\begin{flalign}
\begin{split}
&v_{0}(x)-V(x-\beta)\in H^1     \cap L^1 \quad u_{0}(x)-U(x-\beta)\in H^1     \cap L^1 ,
\end{split}
\label{2.8}\end{flalign}
and \begin{equation}
u_0(0)=0
\end{equation}
as the compatibility condition. Set
\begin{eqnarray*}
(A_{0},B_0)(x):= -\int_{x}^{\infty} (v_{0}(y)-V(y-\beta), u_{0}(y)-U (y-\beta)   )\operatorname{d}y.
\end{eqnarray*}
We further assume that
\begin{eqnarray}
( A_{0} ,B_0 ) \in L^2 .
\label{2.10}\end{eqnarray}

The shift $\beta_{0}$  has the following properties.
\begin{lemma}[\cite{mm1999}]
Under the assumptions   (\ref{2.8})-(\ref{2.10}), the shift $\beta_{0}$ defined by (\ref{2.7}) satisfies
$$\beta_{0}\rightarrow 0  \quad\text {as} \quad \|A_0,B_0\|_{2}\rightarrow 0  \quad \text {and}\quad  \beta\rightarrow +\infty.$$
\end{lemma}
The   main theorem is stated as follows.
\begin{thm}\label{theorem}
 For any   $u_+<0$ and $v_+>0 $, suppose that (\ref{2.8})-(\ref{2.10}) hold.
 Then  there exists a positive constant $\delta_{0}$ such that if $$\|(A_{0},B_{0})\|_{2 }+\beta^{-1}\leq \delta_{0},  $$then the  initial-boundary value problem (\ref{1.1}), (\ref{1.4}) has a unique global solution $(v,u) (x,t)  $, satisfying
\begin{align}\label{2.11}
  v
  &(x,t)-V (x-st+\beta_{0}-\beta)\in C^0([0,+\infty);H^1   ) \cap  L^2([0,+\infty);H^1  ),  \nonumber\\
  u&(x,t)-U(x-st+\beta_{0}-\beta)\in C^0([0,+\infty);H^1  ) \cap  L^2([0,+\infty);H^2 ),
\end{align}
where  $s>0$ is defined by (\ref{1.5}),   and
\begin{flalign}\label{2.12}
\begin{split}
&\sup_{x\in \mathbb{R}_{+}}                       | {(v,u)(x,t)}-(V, U) (x-st+\beta_{0}-\beta) |     \rightarrow 0,  \text{   as } t\rightarrow +\infty.
\end{split}
\end{flalign}
\end{thm}
\begin{remark}
The condition $ v_+-v_-  < C(\gamma - 1)^{-2}$ in \cite{mm1999} is removed.
\end{remark}
\section{Reformulation of the Original Problem}\label{Sec.3}

Set
\begin{align}\label{3.1}
  \phi
  &( x,t)=-\int^\infty_x {}{v}(y,t)-V(y-st+\beta_0-\beta)\operatorname{d}y,  \nonumber\\
  \psi&( x,t)=-\int^\infty_x {}{u}(y,t)-U(y-st+\beta_0-\beta)\operatorname{d}y,
\end{align}
which means that we look for the solution $(v,u)(x,t)$ in the form
\begin{align}\label{3.2}
v&( x,t)=\phi_x (x,t)+  V(x-st+\beta_0-\beta), \nonumber\\
u&( x,t)=\psi_x (x,t)+  U(x-st+\beta_0-\beta).
\end{align}
The initial perturbations $\phi$ and $\psi$    satisfy
\begin{lemma}[\cite{mm1999}]   Under the assumptions (\ref{2.8})-(\ref{2.10}), the initial perturbation $(\phi,\psi)(x,0):=\left(\phi_{0}, \psi_{0}\right)(x) \in H^{2}$ and satisfies
$$
\left\|\left(\phi_{0}, \psi_{0}\right)\right\|_{2} \rightarrow 0 \quad \text { as } \quad\left\|\left(A_{0}, B_{0}\right)\right\|_{2} \rightarrow 0 \text { and } \beta \rightarrow+\infty.
$$
\end{lemma}
Motivated by  \cite{mm1999},  substitute (\ref{3.2}) into (\ref{1.1}) and integrate the resulting system with respect to $x$,  we have
\begin{equation}
\left\{ \begin{array}{ll}
&\phi_{t}-\psi_{x}  =0, \\
&\psi_{t}-f(V) \phi_{x}-\frac{{\psi_{x x}}}{V^{\alpha+1}} =F ,
\end{array} \right.  \label{3.3}
\end{equation}
 with the initial conditions   and Neumann boundary condition:
\begin{align}\label{3.4}
\left(\phi_{0}, \psi_{0}\right)&(x) \in H^{2}, \quad x \geq 0, \nonumber\\
\left.\psi_{x}\right|_{x=0}&=\left.\phi_{t}\right|_{x=0}= -U(st+\beta_0-\beta),     \quad t \geq 0,
\end{align}
where
\begin{align}\label{3.5}
f(V)= -p^{\prime}(V)    + (\alpha+1)\frac{{} s V_{x}}{V^{\alpha+2}}=-p^{\prime}(V)+(\alpha+1)\frac{h(V)  }{V}>0,
\end{align}
\begin{align}\label{3.6}
F&= \frac{{}{u}_{x}}{{}{v}^{\alpha+1}}-\frac{U_{x}}{V^{\alpha+1}} - \frac{\psi_{xx}}{V^{\alpha+1}}+(\alpha+1)\frac{U_x \phi_x }{V^{\alpha+2}}-\left[p({}{v})-p(V)-p'(V)\phi_x\right] \nonumber\\
&=O(1)(|\phi_{x}|^{2}+|\phi_{x}\psi_{xx}|).
\end{align}

 We will seek the solution in the functional space $X_{\delta}(0,T)$ for any $0\leq T < +\infty $,
\begin{align*}
\begin{split}
X_{\delta}(0,T):=&\left\{ (\phi,\psi)\in C ([0,T];H^2)|\phi_{x} \in  L^2(0,T;H^1) ,\psi_{x} \in  L^2(0,T;H^2)\right. \\
& \sup_{0\leq t\leq T}\|(\phi, \psi)(t)\|_2\leq\delta   \},
\end{split}&
\end{align*}
 where ${ \delta} \ll 1$ is small.
\begin{proposition}\label{prposition3.1} (A priori estimate)
Suppose  that $(\phi,\psi) \in X_{\delta}(0,T)$ is the solution of  (\ref{3.3}), (\ref{3.4})  for some time $T>0$. There exists a positive constant $\delta_0  $ independent of  $T$, such that if
$$ \sup_{0\leq t\leq T}\|(\phi, \psi)(t)\|_{{2}} \leq \delta_{ }\leq \delta_{0},$$   for $t \in [0,T]$,
then
\begin{eqnarray*}
\|(\phi, \psi)(t)\|_{{2}}^{2}+    \int_{0}^t     (  \|  \phi_{x}(t)  \|^2_{1}  +   \|\psi_{x}(t)\|_{2}^2   )    \operatorname{d}t       \leq C_{0} ( \|(\phi_{0},\psi_{0})\|_{{2}}^2+   e^{-C_-\beta}),
\end{eqnarray*}
where $C_{0}  >1$ and $C_{-} $ are  positive constants independent of $T$.
\end{proposition}

As long as  Proposition \ref{prposition3.1} is proved,   the local solution $(\phi,\psi)$   can be extend to $T =+\infty. $ We have the following Lemma.

\begin{lemma}\label{lemma 3.2}
  If $(\phi_{0},\psi_{0})\in H^2$, there exists a positive constant $\delta_{1}=\frac{\delta_{0}}{\sqrt{C_{0}}}$, such that if
$$\|(\phi_{0},\psi_{0})\|_{{2}}^2+   e^{-C_-\beta} \leq \delta_{1}^{2},$$ then  the initial-boundary problem  (\ref{3.3}), (\ref{3.4}) has a unique global solution
$(\phi,\psi)\in X_{\delta_{0}}(0,\infty)$  satisfying
\begin{align*}
\sup_{t\geq0}\|(\phi, \psi)(t)\|_{{2}}^{2}+    \int_{0}^\infty   (  \|  \phi_{x}(t)  \|^2_{1}  +   \|\psi_{x}(t)\|_{2}^2   )    \operatorname{d}t       \leq C_{0}  ( \|(\phi_{0},\psi_{0})\|_{{2}}^2+   e^{-C_-\beta}).
\end{align*}
\end{lemma}

\section{A Priori Estimate}\label{Sec.4}

Throughout this section, we assume that the problem $(\ref{3.3}), (\ref{3.4})$ has a solution $(\phi,\psi)\in X_{\delta}(0,T), $ for some $T>0$,

\begin{eqnarray}\label{4.1}
\sup_{0\leq t\leq T}\|(\phi, \psi)(t)\|_{2}\leq \delta.
\end{eqnarray}

It follows from the Sobolev inequality that $\frac{1}{2}v_{+}\leq v \leq \frac{3}{2}v_{-}$, and
\begin{align*}
\sup_{0\leq t\leq T}  \{  \|(\phi, \psi)(t)\|_{L^{\infty}}    +   \|(\phi_{x}, \psi_{x})(t)\|_{L^{\infty}}   \}  \leq  {\delta} .
\end{align*}

\subsection{  Low Order Estimate.}In order to  remove  the condition  $ v_+-v_-< C(\gamma-1)^{-2}  $ in  \cite{mm1999}, we introduce  a new perturbation $(\phi,\Psi)$ instead of $(\phi,\psi)$, where $\Psi$ will be  defined below.

Inspired by \cite{vy2016} and \cite{hh2020}, we introduce a new variable $h$ which depends on $v$ and $u$, i.e., ${}{h}={}{u}-{}{v}^{-(\alpha+1)}{}{v}_{x}$. Through a direct calculation,  $v$ and $h$   satisfy the following system
\begin{equation}
\left\{ \begin{array}{ll}
&{}{v}_t-{}{h}_x=(\frac{{}{v}_{x}}{{}{v}^{\alpha+1}})_{x},\\
&{}{h}_t+{}{p}_x=0.
\end{array} \right.\label{4.2}
\end{equation}
Then the initial-boundary conditions given in (\ref{1.4}) are changed into
\begin{equation*}
\left\{ \begin{array}{ll}
(v, h)(x, 0)= (v_0,u_0   -{}{v_{0}}^{-(\alpha+1)}{}{v}_{0x})(x) \longrightarrow (v_+,u_+), ~x\to +\infty,\\
h(0,t)={u}(0,t)-{}{v(0,t)}^{-(\alpha+1)}{}{v_{x}(0,t)}=-{}{v(0,t)}^{-(\alpha+1)}{}{v_{x}(0,t)}, t\in\mathbb{R_+}.
\end{array} \right.
\end{equation*}

Let $H=U-V^{-(\alpha+1)}V_{x}$. Then (\ref{2.1}) is equivalent to
\begin{equation}
\left\{ \begin{array}{rl}
V_{t}-H_{x}=&\left(\frac{V_{x}}{V^{\alpha+1}}\right)_{x},\\
H_{t}+p(V)_{x}=&0,\\
(V ,H)(-\infty)=&(v_-,0),\quad (V ,H)(+\infty)=(v_+,u_+).
\end{array} \right.\label{4.3}
\end{equation}

We define
\begin{eqnarray}
-\int_{x}^\infty({}{h}-H)\operatorname{d}x=\Psi.
\label{4.4}\end{eqnarray}
Substituting (\ref{4.3}) from (\ref{4.2}) and integrating the resulting system with respect to $x$, we have from (\ref{4.4}), $(\ref{3.1})_1$ that
\begin{equation}
\left\{ \begin{array}{ll}
&\phi_t- \Psi_x-\frac{\phi_{xx}}{V^{\alpha+1}}+(\alpha+1)\frac{V_x \phi_x }{V^{\alpha+2}}=G,\\
&\Psi_t+p'(V)\phi_x=-p( {v}|V) ,
\end{array} \right.  \label{4.5}
\end{equation}
where
\begin{eqnarray*}
G=\frac{{}{v}_{x}}{{}{v}^{\alpha+1}}-\frac{V_{x}}{V^{\alpha+1}} - \frac{\phi_{xx}}{V^{\alpha+1}}+(\alpha+1)\frac{V_x \phi_x }{V^{\alpha+2}},
\end{eqnarray*}
\begin{eqnarray*}
p({v}|V)=\left(p({}{v})-p(V)\right)-p'(V)\phi_x,
\end{eqnarray*}
with the  initial data
\begin{eqnarray*}
 \phi  (x,0) \in H^{2}, \quad \Psi (x,0) \in H^{1},
\end{eqnarray*}
and boundary data
\begin{eqnarray*}
 \Phi  (0,t)=-\int_{x}^{\infty}           [ {u} (y,0)-U(y  +\beta_{0}-\beta  )         ]  \operatorname{ d }y +\left( V^{-(\alpha+1)}-{}{v}^{-(\alpha+1)}\right) (x,0).
\end{eqnarray*}

\begin{lemma} (\cite{hh2020})\label{lemma 4.1}
Under the assumption of (\ref{4.1}), it holds that
\begin{flalign*}
\begin{split}
&p({}{v}|V)\leq C \phi_{x}^{2},\\
&|p({}{v}|V)_{x}|\leq C (|\phi_{xx}\phi_{x}|+|V_x|\phi_{x}^{2}),\\
&|G|\leq C (|\phi_{xx}\phi_{x}|+|V_x|\phi^{2}_{x}).
\end{split}
\end{flalign*}
\end{lemma}

In addition, some boundary estimates are given as follows.
 \begin{lemma}\label{lemma 4.2}
 Under the same assumptions of Proposition \ref{prposition3.1}, for $0 \leq t \leq T$, it holds that:
\begin{align}\label{4.6}
 &  \left|\int_{0}^{t}(\phi \Psi)|_{x=0} \operatorname{d}t \right|  \leq C e^{-C_{-} \beta},   \quad\, \,\, \,  \,\, \,    \left| \int_{0}^{t}\left(\phi  \phi_{x}\right)  |_{x=0} \operatorname{d}t \right|  \leq C e^{-C_{-} \beta},\\
 \label{4.7}
  &  \left|\int_{0}^{t}(\phi_{x} \phi_{t})|_{x=0} \operatorname{d}t \right|  \leq C e^{-C_{-} \beta}, \quad \, \,\, \,   \left| \int_{0}^{t}\left(\psi_{x}  \psi_{t}\right)  |_{x=0} \operatorname{d}t \right|  \leq C e^{-C_{-} \beta},\\
  \label{4.8}
    & \left|\int_{0}^{t}(\psi_{x} \psi_{xx})|_{x=0}  \operatorname{d}t \right|  \leq C e^{-C_{-} \beta}, \quad  \left| \int_{0}^{t}\left(\psi_{xt}  \psi_{xx}\right)  |_{x=0}  \operatorname{d}t \right|  \leq C e^{-C_{-} \beta},
\end{align}
and
\begin{eqnarray}  \label{4.9}
  \|    {\Psi}_{0}  \|_{1}^{2}   &\leq&  \|    {\psi}_{0}  \|_{1}^{2}+    C  \|    {\phi}_{0}   \|_{2}^{2},
    \quad\| \psi \|^{2}  \leq      \| \Psi \| ^2+ C\|\phi \|_{ 1}^2,  \nonumber    \\
    \| \psi_{x} \|^{2} & \leq &     \| \Psi_{x} \| ^2+ C\|\phi_{x} \|_{ 1}^2,
\end{eqnarray}
where  $C_-=\frac{v_{-}^{\alpha+1}}{   s_{ }} |p'(v_-)+s_{} ^2|>0$.
\end{lemma}

\begin{proof}
Note that
\begin{eqnarray} \label{4.10}
  \notag \Psi (x,t)&=&-\int_{x}^{\infty}           [ {u} (y,t)-U(y-st +\beta_{0}-\beta  )         ]  \operatorname{ d }y\\\notag &&+\left( V^{-(\alpha+1)}-{}{v}^{-(\alpha+1)}\right) (x,t) \\\notag
    &=:& \psi (x,t)  +p(x,t )\leq \psi (x,t)  + C \phi_{ x} (x,t),\\
     \psi (x,t)&=& \Psi (x,t)  -p(x,t )\leq \psi (x,t)  + C \phi_{ x} (x,t),
\end{eqnarray}
one have (\ref{4.9}) from  (\ref{4.10}) immediately. Motivated by  \cite{mm1999}, we have
\begin{align}  \label{4.11}
|\psi (0,t)|\leq C,   \quad  |\phi_{x} (0,t)|\leq C,   \quad |\phi_{ } (0,t)|\leq C e^{-C_{-} \beta}e^{-C_{-} st}.
\end{align}
Combining (\ref{4.10}) and (\ref{4.11}), we have (\ref{4.6}). The  estimates (\ref{4.7}) and (\ref{4.8}) can  be found in  \cite{mm1999}. Thus the proof is completed.
\end{proof}

\begin{lemma}\label{lemma4.3}
Under the same assumptions of Proposition \ref{prposition3.1}, it holds that
\begin{align*}
\begin{split}
&\|(\phi,\Psi)\|_{ }^2(t)+ \int_0^t\int_{0}^{\infty}   \left(\frac{1}{p'(V)}\right)_t \Psi^2\operatorname{d} x \operatorname{d}t+\int_0^t \| \phi_x\|^2 \operatorname{d}t\\
\leq&~ C_{}\|(\phi_0,\Psi_0)\|_{ }^2+C{\delta}  \int_0^t \| \phi_{xx}\|^{2}\operatorname{  d}t + C_{} e^{-C_-\beta}.
\end{split}
\end{align*}
\end{lemma}
\begin{proof}
Multiply \ $(\ref{4.5})_1 $  and  $(\ref{4.5})_2$  by $\phi$ and $\frac{\Psi}{-p'(V)}$  respectively, sum them up, and  integrate the result with  respect to $t$ and $x$ over $ [0,t]\times [0,\infty)  $. We have
\begin{align}\label{4.12}
&\frac{1}{2}    \int_{0}^{\infty}      \left(\phi^2-\frac{\Psi^2}{p'(V)}\right)   \operatorname{d  }x
+    \int_0^t\int_{0}^{\infty}   \left\{ \frac{1}{2} \left(\frac{1}{p'(V)}\right)_t \Psi^2
+                  \frac{\phi_{x}^2}{V^{\alpha+1}} \right\}  \operatorname{d}x \operatorname{d} t \nonumber \\
=&           \int_0^t\int_{0}^{\infty}       G_{ }\phi  \operatorname{d}x \operatorname{d}  t
+            \int_0^t\int_{0}^{\infty}          \frac{p(v|V)\Psi}{p'(V)}  \operatorname{d}x \operatorname{d}t
\nonumber  \\
&-           \int_0^t         (\phi\Psi+(V^{-(\alpha+1)})\phi\phi_{x} )|_{x=0}       \operatorname{  d }t+\frac{1}{2}    \int_{0}^{\infty}      \left(\phi^2-\frac{\Psi^2}{p'(V)}\right)\Big|_{t=0}   \operatorname{d  }x \nonumber  \\
= :&\sum_{i=1}^4 A_i.
\end{align}

Utilize   Lemma \ref{lemma 4.1}, we can get

\begin{align}\label{4.13}
&|A_1+A_2| \nonumber  \\
\leq& C \left(\int_0^t\int_{0}^{\infty}   \left|\phi_x \phi_{xx} \phi \right| +  \left| V_x \phi_x^2 \phi\right|  +     \left| \Psi \phi_x^2 \right|  \operatorname{d}x \operatorname{d}t\right) \nonumber  \\
\leq& C \int_0^t    \|\phi\|_{L^\infty}    \int_{0}^{\infty}   \left|\phi_x \phi_{xx}  \right| \operatorname{d}x \operatorname{d}t+ C \int_0^t     (\|\phi\|_{L^\infty}+ \|\Psi\|_{L^\infty})  \int_{0}^{\infty}   \phi_x^2  \operatorname{d}x \operatorname{d}t\nonumber  \\
\leq&   C (\|\phi\|_{2}+  \|\psi\|_{1}   ) \int_0^t        \|\phi_x\|^2 +\|\phi_{xx}\|^2  \operatorname{  d}t\nonumber  \\
\leq & C \delta \int_0^t        \|\phi_x\|^2 +\|\phi_{xx}\|^2  \operatorname{  d}t.
\end{align}

With the help of  Lemma \ref{4.2}, one has
\begin{align}\label{4.14}
\begin{split}
|A_3|\leq& C    e^{-C_- \beta}.
\end{split}&
\end{align}

Taking $\delta$  sufficiently small, using (\ref{4.12})--(\ref{4.14}), we get Lemma \ref{lemma4.3}.
\end{proof}

\begin{lemma}\label{lemma4.4}
Under the same assumptions of Proposition \ref{prposition3.1}, it holds that
\begin{eqnarray*}
\|(\phi,\Psi)(t)\|_{ 1}^2+\int_0^t\|  \phi_x  \|_{ 1}^2\operatorname{d} t\leq C_{ }\|(\phi_0,\Psi_0)\|_{ 1}^2 + C_{ } e^{-C_-\beta}.
\end{eqnarray*}
\end{lemma}

\begin{proof}
Multiply $ (\ref{4.5})_1 $ and $  (\ref{4.5})_2 $ by $-\phi_{xx}$  and $\frac{\Psi_{xx}}{p'(V)}$ respectively,   sum over the result, integrate the result with  respect to $t$ and $x$ over $ [0,t]\times [0,\infty) $. We have

\begin{align}
& \frac{1}{2}  \int_{0}^{\infty}  \left(\phi_x^2-\frac{\Psi_x^2}{p'(V)}\right)   \operatorname{d}x    + \int_0^t\int_{0}^{\infty} \left\{ \frac{1}{2}  \left(\frac{1}{p'(V)}\right)_t\Psi_x^2  + \frac{\phi^2_{xx}}{V^{\alpha+1}}\right\} \operatorname{d}x \operatorname{d} t  \nonumber \\
=&\frac{1}{2}  \int_{0}^{\infty}  \left(\phi_x^2-\frac{\Psi_x^2}{p'(V)}\right)\Big|_{t=0}   \operatorname{d}x   \nonumber \\
 &-\int_0^t\int_{0}^{\infty}   \left[G-(\alpha+1)\frac{V_x }{V^{\alpha+2}}  \phi_x   \right]\phi_{xx} \operatorname{d}x \operatorname{d} t \nonumber \\
&-  \int_0^t\int_{0}^{\infty}     \left(\frac{1}{p'(V)}\right)_{x}  p'(V) \Psi_{x} \phi_{x} \operatorname{d}x \operatorname{d} t   \nonumber  \\
  &-             \int_0^t         \left(\phi_{t}\phi_{x}-\phi_{x}\Psi_{x} -\frac{\Psi_{t}\Psi_{x}}{p'(V)}-\frac{p(v|V)}      {p'(V)}\Psi_{x}\right)\Big|_{x=0}       \operatorname{  d }t \nonumber  \\
 &+  \int_0^t\int_{0}^{\infty}      \frac{1}{p'(V)} p(v|V)_{x}\Psi_{x} \operatorname{d}x \operatorname{d} t  \nonumber  \\
= :  &\frac{1}{2}  \int_{0}^{\infty}  \left(\phi_x^2-\frac{\Psi_x^2}{p'(V)}\right)\Big|_{t=0}   \operatorname{d}x +\sum_{i=1}^{4} B_i.
\label{4.15}\end{align}
Now we estimate $B_i$ term by term. The Cauchy inequality indicates that
\begin{align}
|B_1|&    \leq     C  \int_0^t\int_{0}^{\infty}   (|\phi_{xx}\phi_x|+|V_x\phi^{2}_x|)|{ \phi_{xx}}| +      | \phi_x     \phi_{xx} |               \operatorname{d} x \operatorname{d}t \nonumber   \\
  &\leq (C \delta + \varepsilon )\int_0^t  \|\phi_{xx}\|^{2}  \operatorname{  d}t  + C_{\varepsilon} \int_0^t   \|\phi_{x}\|^2\operatorname{  d}t ,
\end{align}
and
\begin{align}
|B_2|\leq&  \int_0^t  \int_{0}^{\infty}  \left|p'(V)\Psi_x \phi_x\left(\frac{1}{p'(V)}\right)_x  \right|  \operatorname{d} x \operatorname{d}t \nonumber \\
\leq&\frac{1}{4} \int_0^t\int_{0}^{\infty}\left(\frac{1}{ p'(V)}\right)_t\Psi_x ^2 \operatorname{d} x \operatorname{d}t +C \int_0^t   \|\phi_x\| ^2 \operatorname{  d}t.
\end{align}
Making use of the estimate (\ref{4.7}) for the boundary, we have
\begin{align}
 B_3 &=          -\int_0^t         \left(\phi_{t}\phi_{x}-\phi_{x}\Psi_{x} -\frac{\Psi_{t}\Psi_{x}}{p'(V)}-\frac{p(v|V)}      {p'(V)}\Psi_{x}\right)\Big|_{x=0}       \operatorname{  d }t \nonumber  \\
 &=         - \int_0^t         (\phi_{t}\phi_{x} )|_{x=0} \leq C    e^{-C_- \beta}.
\end{align}
By (\ref{4.9}) and the Sobolev inequality, we obtain
\begin{align}
|B_4|\leq& \int_0^t\int_{0}^{\infty}     \left| \frac{1}{p'(V)} p({v}|V)_{x}\Psi_{x}\right| \operatorname{d}x \operatorname{d} t \nonumber \\
\leq & C\int_0^t\int_{0}^{\infty}     \left| (\phi_{x}\phi_{xx}+V_x\phi_{x}^{2}) \Psi_{x}       \right| \operatorname{d}x \operatorname{d} t  \nonumber  \\
\leq&   C\int_0^t\int_{0}^{\infty}    \left\{ \left| (\phi_{x}\phi_{xx}+V_x\phi_{x}^{2}) \psi_{x}       \right| +   \left| (\phi_{xx}\phi_{xx}+V_x\phi_{x}\phi_{xx} ) \phi_{x}       \right| \right\}\operatorname{d}x\operatorname{d} t \nonumber \\
\leq&   C (\|\phi\|_{2}+  \|\psi\|_{2}   ) \int_0^t        \|\phi_x\|^2 +\|\phi_{xx}\|^2  \operatorname{  d}t \nonumber \\
\leq&   C\delta \int_0^t    (\|\phi_{xx}\|^2+\|\phi_{x}\|^2 )    \operatorname{ d} t.
\label{4.19}\end{align}

From (\ref{4.15})--(\ref{4.19}), we get
\begin{align*}
\begin{split}
& \frac{1}{2} \int_{0}^{\infty}  \left(\phi_x^2-\frac{\Psi_x^2}{p'(V)}\right)    \operatorname{d  }x    + \frac{1}{4}\int_0^t\int_{0}^{\infty}    \left[\left(\frac{1}{p'(V)}\right)_t\Psi_x^2  +\frac{\phi^2_{xx}}{V^{\alpha+1}} \right] \operatorname{d}x \operatorname{d} t \\
\leq&  (C +C\delta+C_{\varepsilon} )\int_0^t   \|\phi_{x}\|^2 \operatorname{  d}t
 +(C\delta+\varepsilon )\int_0^t     \|\phi_{xx}\|^2  \operatorname{  d}t\\
& +  C_{ } e^{-C_-\beta}+C_{ }    \left(   \|\phi_{0x}\|^2   +\|\Psi_{0x}\|^2 \right)   .
\end{split}&
\end{align*}
Choosing  $\varepsilon$ sufficiently small,   together with Lemma \ref{lemma4.3}, we complete  the proof of Lemma \ref{lemma4.4}.
\end{proof}

\begin{lemma}\label{lemma4.5}
Under the same assumptions of Proposition \ref{prposition3.1}, it holds that
\begin{eqnarray*}
\int_0^t    \|\Psi_{x}(t)\|_{  }^2   \operatorname{d}t  \leq   C  \|(\phi_0,\Psi_0)\|_{ 1}^2 +    C e^{-C_-\beta}.
\end{eqnarray*}
\end{lemma}
\begin{proof}
Multiply $(\ref{4.5})_1 $  by $\Psi_{x}$  and make use of $(\ref{4.5})_2$. We get
\begin{align}\label{4.20}
\Psi_{x}^{2}=&(\phi\Psi_{x})_{t}+[\phi (p(v)-p({V}{}))]_{x}-\phi_{x} (p( v)-p({V}{})) \nonumber  \\
&-\frac{\Psi_{x}\phi_{xx}}{V^{\alpha+1}}-\Psi_{x}\left[G-(\alpha+1)\frac{V_x\phi_x}{V_{\alpha+2}}\right].
\end{align}
Integrate $ (\ref{4.20})$ with  respect to $t$ and $x$ over $ [0,t]\times [0,\infty)$. We have
\begin{align}\label{4.21}
& \int_{0}^{t}      \|\Psi_{x}\|^{2}  \operatorname{    d}t \nonumber \\
=&-\int_{0}^{\infty} \phi\Psi_{x}|_{t=0}\operatorname{d }x+\int_{0}^{t}\int_{0}^{\infty} -\Psi_{x}\left[G-(\alpha+1)\frac{V_x\phi_x}{V_{\alpha+2}}\right]\operatorname{d}x \operatorname{d}t
\nonumber  \\
&+\int_{0}^{\infty}\phi\Psi_{x}\operatorname{d} x-\int_{0}^{t}\int_{0}^{\infty} \frac{\Psi_{x}\phi_{xx}}{V^{\alpha+1}}\operatorname{d}x  \operatorname{d}t \nonumber \\
&-\int_{0}^{t}\int_{0}^{\infty}\phi_{x}\left(p( {v})-p(V)\right)       \operatorname{d}x \operatorname{d}t-\int_{0}^{t} \phi (p(v)-p({V}{}))|_{x=0}     \operatorname{  d}t \nonumber  \\
=:&-\int_{0}^{\infty} \phi\Psi_{x}|_{t=0}\operatorname{d}x+\sum_{i=1}^5 H_i.
\end{align}
We estimate $H_i$ term by term. By the Cauchy inequality, it holds that
\begin{align}
H_1&\leq C \int_{0}^{t}\int_{0}^{\infty}    \Psi_{x}(|\phi_{x}\phi_{xx}|+|V_{x}\phi_{x}|)         \operatorname{d}x \operatorname{d}t \nonumber \\
&\leq  \varepsilon \int_{0}^{t}  \|\Psi_{x}\|^{2}  \operatorname{  d}t+ C_{\varepsilon} \int_{0}^{t}   ( \|\phi_{xx}\|^{2} + \|\phi_{x} \|^{2}   )      \operatorname{  d}t.
\end{align}
In addition, it is straightforward to imply that
\begin{align}
&H_2+H_3+H_4 \nonumber  \\
\leq &  \| (\phi^{}, \Psi_{x}) \|^{2}+\varepsilon\int_{0}^{t} \|\Psi_{x}\|^{2}\operatorname{  d}t+C_{\varepsilon}\int_{0}^{t} \|\phi_{xx}\|^2       \operatorname{ d}t+C\int_{0}^{t}\|\phi_{x}\|^2 \operatorname{  d}t.
\end{align}
Making use of the estimate (\ref{4.6}) for the boundary, we have

\begin{align}\label{4.24}
\begin{split}
H_5&= -\int_{0}^{t} \phi (p(v)-p({V}{}))|_{x=0}     \operatorname{  d}t \leq C \int_{0}^{t} \phi  \phi_{x} |_{x=0}     \operatorname{  d}t  \leq Ce^{-C_-\beta}.\\
\end{split}&
\end{align}
Collecting   (\ref{4.21})-(\ref{4.24}) and using Lemma \ref{lemma4.4}, we complete the proof of Lemma \ref{lemma4.5}.
\end{proof}

Combining Lemma \ref{lemma4.3}-Lemma \ref{lemma4.5}, we  obtain the following low order estimates
\begin{align*}
\|(\phi,\Psi)\|_{ 1}^2(t)+ \int_0^t  \| \Psi_x\|^2  \operatorname{d}t+\int_0^t \| \phi_x\|_{1}^2 \operatorname{d}t\leq C_{}\|(\phi_0,\Psi_0)\|_{ 1}^2 + C_{} e^{-C_-\beta},
\end{align*}
which can be rewritten by the variables $\phi$ and $\psi$ as
\begin{lemma}\label{lemma4.6}
Under the same assumptions of Proposition \ref{prposition3.1}, it holds that
\begin{align*}
\begin{split}
&(\| \phi \|_{ 1}^2 + \| \psi \| ^2)(t)+\int_0^t  \| \psi_x\|^2  \operatorname{d}t+\int_0^t \| \phi_x\|_{1}^2 \operatorname{d}t\leq C_{}\|     \phi_0 \|_{2}^2 +C_{}\| \psi_0 \|_{ 1}^2 + C_{} e^{-C_-\beta}.
\end{split}
\end{align*}
\end{lemma}
\subsection{High Order Estimate.}

Since the second derivative of $\Psi$ on the boundary is unknown, we turn to the original equation \eqref{3.3} to study the higher order estimates.
\begin{lemma}\label{lemma4.7}
Under the same assumptions of Proposition \ref{prposition3.1}, it holds that
\begin{align}\label{4.25}
\begin{split}
&  \| \psi_{x} \|  ^2 (t)  +\int_0^t  \| \psi_{xx}\| ^2  \operatorname{d}t\leq C_{}\|     \phi_0 \|_{2}^2 +C_{}\| \psi_0 \|_{ 1}^2 + C_{} e^{-C_-\beta}.
\end{split}
\end{align}
\end{lemma}
\begin{proof}  Multiplying   $(\ref{3.3})_{2}$ by $-\psi_{x x}$,  integrating the result with  respect to $t$ and $x$ over $ [0,t]\times [0,\infty) $ gives
\begin{align} \label{4.26}
&   \frac{1}{2}\| \psi_{x}\|^{2}  (t)
+\int_{0}^{t}\int_{0}^{\infty} \frac{{\psi_{x x}^{2}}}{V^{\alpha+1}} \operatorname{d}x \operatorname{d}t \nonumber \\
=&\frac{1}{2}\| \psi_{0x}\|^{2}  -\int_{0}^{t}  \left\{\psi_{x} \psi_{t}\right\}|_{x=0} \operatorname{  d}t
-\int_{0}^{t}\int_{0}^{\infty} f(V) \phi_{x} \psi_{x x} \operatorname{d}x \operatorname{d}t \nonumber \\
&-\int_{0}^{t}\int_{0}^{\infty} F \psi_{x x} \operatorname{d}x \operatorname{d}t
 \nonumber  \\
=:&\frac{1}{2}\| \psi_{0x}\|^{2} +\sum_{i=1}^3 M_i.
\end{align}
Making use of the estimate (\ref{4.7}) for the boundary, we have
\begin{align} \label{4.27}
M_{1}  \leq C_{} e^{-C_-\beta}.
\end{align}
The Cauchy inequality implies that
\begin{align}
M_{2}  \leq \varepsilon \int_{0}^{t} \|\psi_{x x}\|^{2}\operatorname{  d}t+ C_{\varepsilon}\int_{0}^{t}  \|\phi_{x}\|^{2}\operatorname{  d}t.
\end{align}
By (\ref{3.6}) and the Sobolev inequality, yields
\begin{align} \label{4.29}
M_{3} &\leq C \int_{0}^{t}\int_{0}^{\infty}\left(\left|\phi_{x}\right|^{2}  +\left|\phi_{x}\right|\left|\psi_{x x}\right|\right)\left|\psi_{x x}\right| \operatorname{d} x \operatorname{d}t \nonumber  \\
&\leq C \int_{0}^{t}\int_{0}^{\infty}\left|\phi_{x}\right|\left(\left|\phi_{x}\right|^{2}+\left|\psi_{x x}\right|^{2}\right) \operatorname{d}x\operatorname{d}t \nonumber  \\
&\leq C   \delta \int_{0}^{t} \left(\left\|\phi_{x}\right\|^{2}+\left\|\psi_{x x}\right\|^{2}\right) \operatorname{  d}t.
\end{align}
Substituting (\ref{4.27})-(\ref{4.29})  into $(\ref{4.26}) $ and using Lemma \ref{lemma4.6}, we obtain (\ref{4.25}).
\end{proof}

\begin{lemma}\label{lemma4.8}
Under the same assumptions of Proposition \ref{prposition3.1}, it holds that
\begin{align}\label{4.30}
\begin{split}
&\| \phi_{xx} \| ^2   +\int_0^t  \| \phi_{xx}\|_{}^2  \operatorname{d}t \leq C_{}\|     \phi_0 \|_{2}^2 +C_{}\| \psi_0 \|_{ 1}^2 + C_{} e^{-C_-\beta}+ C\delta \int_{0}^{t}\left\|\psi_{xx x} \right\|^2  \operatorname{d}t.
\end{split}
\end{align}
\end{lemma}

\begin{proof}
 Differentiating    $(\ref{3.3})_{1}$ with  respect to $x$,  using  $(\ref{3.3})_{2}, $ we have
\begin{align}\label{4.31}
\begin{split}
\frac{  \phi_{x t}}{V^{\alpha+1}}+f(V) \phi_{x}=\psi_{t}-F.
\end{split}
 \end{align}
Differentiating    $(\ref{4.31})$ in respect of $x$ and multiplying  the derivative  by $\phi_{x x}$,  integrating the result in respect of $t$ and $x$ over $ [0,t]\times [0,\infty) $, using (\ref{2.3}), one has
\begin{align} \label{4.32}
 &\frac{1}{2} \int_{0}^{\infty}  \frac{ \phi_{x x}^{2}}{  V^{  \alpha+1  }}  \operatorname{d  }x
 +\int_{0}^{t}\int_{0}^{\infty}\left(f(V)-\frac{(\alpha+1)h(V)}{2 V}\right) \phi_{x x}^{2}\operatorname{d}x \operatorname{d}t \nonumber  \\
 =& \frac{1}{2} \int_{0}^{\infty}  \frac{ \phi_{x x}^{2}}{  V^{  \alpha+1  }}  \Big|_{t=0}\operatorname{d  }x-\int_{0}^{\infty}   \psi_{x } \phi_{x x}    \Big|_{t=0}\operatorname{d  }x+\int_{0}^{\infty}   \psi_{x } \phi_{x x}   \operatorname{d  }x\nonumber  \\
&+\int_{0}^{t}  \psi_{x } \psi_{x x} \Big|_{x=0}\operatorname{  d}t
+\int_{0}^{t}        \| \psi_{x x}\|^{2}\operatorname{    d}t-\int_{0}^{t}\int_{0}^{\infty}F_{x} \phi_{x x}\operatorname{d}x\operatorname{d}t
\nonumber  \\
&+(\alpha+1)\int_{0}^{t}\int_{0}^{\infty}\frac{{} V_{x}}{V^{\alpha+2}} \phi_{x t} \phi_{x x}\operatorname{d}x\operatorname{d}t-\int_{0}^{t}\int_{0}^{\infty}f(V)_{x} \phi_{x} \phi_{x x}\operatorname{d}x\operatorname{d}t\nonumber  \\
=:&\frac{1}{2} \int_{0}^{\infty}  \frac{ \phi_{x x}^{2}}{  V^{  \alpha+1  }}  \Big|_{t=0}\operatorname{d }x-\int_{0}^{\infty}   \psi_{x } \phi_{x x}    \Big|_{t=0}\operatorname{d  }x+\sum_{i=1}^6 N_i.
\end{align}
By (\ref{2.5}) and (\ref{3.5}), one has
\begin{align} \label{4.33}
f(V)-\frac{(\alpha+1)h(V)}{2 V}\geq -p'(v_+)>0.
\end{align}
The Cauchy inequality yields
\begin{align}
N_{1}\leq  \varepsilon  \|\phi_{x x}\|^{2}   +C_\varepsilon    \|\psi_{x }\|  ^{2}.
\end{align}
Making use of the estimate (\ref{4.8}) for the boundary, it follows that
\begin{align}
N_{2}  \leq C_{} e^{-C_-\beta}.
\end{align}
$ N_{3}  $ can be controlled by (\ref{4.25}). By  the Cauchy inequality, we have
\begin{align*}
\begin{split}
|N_{4}| \leq &  \varepsilon \int_{0}^{t}  \|\phi_{x x}\|^{2} \operatorname{  d}t  +    C_{\varepsilon}   \int_{0}^{t}\left\|F_{ x}\right\|^{2} \operatorname{  d}t. \\
\end{split}
\end{align*}
Using
\begin{align*}
\begin{split}
\left\|F_{x}\right\|^{2} & \leq C \int_{0}^{\infty}\left(\phi_{x}^{4}+\phi_{x}^{2} \phi_{x x}^{2}+\psi_{x x}^{2} \phi_{x x}^{2}+\psi_{x x x}^{2} \phi_{x}^{2}+\phi_{x}^{2} \psi_{x x}^{2}\right)\operatorname{d}x \\
& \leq C \delta\left(\left\|\phi_{x}\right\|_{1}^{2}+\left\|\psi_{x}\right\|_{2}^{2}\right),
\end{split}
\end{align*}
we have the estimate of $N_4$
\begin{align}
\begin{split}
|N_{4}| \leq &  \varepsilon \int_{0}^{t}  \|\phi_{x x}\|^{2} \operatorname{  d}t  +  C_{\varepsilon} \delta    \int_{0}^{t}  \left(\left\|\phi_{x}\right\|_{1}^{2}+\left\|\psi_{x}\right\|_{2}^{2}\right)\operatorname{  d}t. \\
\end{split}
\end{align}
The Cauchy inequality yields
\begin{align}
\begin{split}
|N_{5}| \leq C & \int_{0}^{t}\int_{0}^{\infty} \left|\frac{{} V_{x}}{V^{\alpha+2}} \psi_{x x} \phi_{x x}\right| \operatorname{d}x \operatorname{d}t\leq \varepsilon \int_{0}^{t}  \|\phi_{x x}\|^{2} \operatorname{  d}t  +    C_{\varepsilon}   \int_{0}^{t}\left\|\psi_{x x}\right\|^{2} \operatorname{  d}t, \\
\end{split}
\end{align}
\begin{align} \label{4.38}
\begin{split}
|N_{6}| \leq &  \varepsilon \int_{0}^{t}  \|\phi_{x x}\|^{2} \operatorname{  d}t  +    C_{\varepsilon}   \int_{0}^{t}\left\|\phi_{ x}\right\|^{2} \operatorname{  d}t. \\
\end{split}
\end{align}
Choosing  $\varepsilon$  small,  substituting  (\ref{4.33})-(\ref{4.38})  into (\ref{4.32})   and  using  Lemma \ref{lemma4.6}, Lemma \ref{lemma4.7},    we have (\ref{4.30}).
\end{proof}

On the other hand, differentiating the second equation of (\ref{3.3}) with respect to $x$, multiplying the derivative by $-\psi_{x x x}$, integrating the resulting equality over $[0, \infty) \times[0, t]$, using Lemma \ref{lemma4.6}-Lemma \ref{lemma4.8}, we can get the highest order estimate in the same way, which is listed as follows  and the proof is omitted.

\begin{lemma}\label{lemma4.9}
Under the same assumptions of Proposition \ref{prposition3.1}, it holds that
\begin{align}\label{4.39}
\begin{split}
&  \| \psi_{xx} (t)\| ^2+\int_0^t  \| \psi_{xxx}\| ^2  \operatorname{d}t \leq   C_{}\|   (  \phi_0 , \psi_0 )\|_{ 2}^2 + C_{} e^{-C_-\beta} .
\end{split}
\end{align}
\end{lemma}
Finally, Proposition \ref{prposition3.1} is obtained by Lemma \ref{lemma4.5}-Lemma \ref{lemma4.9}.

\section{Proof of Theorem \ref{theorem} }\label{Sec.5}

Now we turn to the proof of main theorem, i.e., Theorem \ref{theorem}. It is straightforward to imply (\ref{2.11}) from Lemma \ref{lemma 3.2}.
It remains to show  (\ref{2.12}). We will use  the following useful lemma.
\begin{lemma}(\cite{mn1985})\label{lemma5.1}
 Assume that the function $f(t) \geq 0\in L^1(0, +\infty) \cap  BV(0, +\infty) $. Then it holds that $f(t) \rightarrow0$ as $t \rightarrow \infty$.
 \end{lemma}
\begin{proof}  {\bf (Proof of Theorem \ref{theorem}.)}\,
  Differentiating the first equation of (\ref{3.3}) with respect to $x$, multiplying the
 resulting equation by $\phi_{x}$,  and integrating  on $(0,\infty)$, we have
 \begin{equation*}
  \left|\frac{\operatorname{ d}}{\operatorname{ d}t}\left(\|\phi_{x}\|^{2}\right)\right|\leq C(\|\phi_{x}\|^{2} +\|\psi_{xx}\|^{2})  .
 \end{equation*}
 Using Lemma \ref{lemma 3.2}, we have
 \begin{equation*}
 \int_{0}^{\infty} \left|\frac{\operatorname{ d}}{\operatorname{ d}t}\left(\|\phi_{x}\|^{2}\right)\right| \operatorname{ dt} \leq C\left\{\left\|\left(\phi_{0}, \psi_{0}\right)\right\|_{2}^{2}+e^{-c_{-} \beta}\right\}\leq C,
 \end{equation*}
which implies $\|\phi_{ x}\|^{2}\in L^1(0, +\infty) \cap  BV(0, +\infty)$. By Lemma \ref{lemma5.1}, we have
 \begin{equation*}
    \|\phi_{ x}\|\rightarrow0 \quad   \text{as} \quad   t\rightarrow+\infty.
 \end{equation*}
Since $\|\phi_{xx}\|$ is bounded, the Sobolev inequality implies that
 \begin{eqnarray*}
 \|{}{v}-V\|_{\infty}^{2}=\|\phi_{x}\|_{\infty}^{2} \leq 2\| \phi_{x}(t)\|_{}   \| \phi_{xx}(t)\|_{} \rightarrow  0.
\end{eqnarray*}
 Similarly, we have
 \begin{eqnarray*}
 \|{}{u}-U\|_{\infty}^{2}=\|\psi_{x}\|_{\infty}^{2} \leq 2\| \psi_{x}(t)\|_{}   \| \psi_{xx}(t)\|_{} \rightarrow  0.
\end{eqnarray*}
Therefore, the proof of Theorem \ref{theorem} is completed.
\end{proof}


\begin{thebibliography}{}

          % and use \bibitem to create references.


          % \bibitem{} H. Freistuhler, D. Serre, $L^1$  stability of shock waves in scalar viscous conservation laws. Comm. Pure Appl. Math. 51 (1998), no. 3, 291-301.

          \bibitem{fs1998} H. Freistuhler, D. Serre, {\em $L^1$  stability of shock waves in scalar viscous conservation laws}, Comm. Pure Appl. Math., 51(3): 291--301,1998.


           \bibitem{g1986} J. Goodman, {\em Nonlinear asymptotic stability of viscous shock profiles for conservation laws}, Arch. Rational Mech. Anal., 95(4): 325--344,1986.

          \bibitem{hh2020} L. He, F. Huang,    {\em Nonlinear stability of large amplitude viscous shock wave for general viscous gas},  J. Differential Equations, 269(2):1226--1242,2020.

           \bibitem{hm2009} F. Huang, A. Matsumura, {\em Stability of a composite wave of two viscous shock waves for full compressible
		Navier-Stokes equation},  Comm. Math. Phys., 289(3): 841--861,2009.


\bibitem{hlz2017} J. Humpherys, G. Lyng, K. Zumbrun, {\em Multidimensional stability of large-amplitude Navier-Stokes shocks},   Arch. Ration. Mech. Anal., 226(3): 923--973,2017.

    \bibitem{km1985} S. Kawashima, A. Matsumura, {\em Asymptotic stability of traveling wave solutions of systems for one-dimensional gas motion},  Commun. Math. Phys., 101(1): 97--127,1985.

		\bibitem{l1997} T. Liu,  {\em Pointwise convergence to shock waves for viscous conservation laws}, Comm. Pure Appl. Math., 50(11): 1113--1182,1997.


\bibitem{lz2009} T. Liu, Y. Zeng, {\em  Time-asymptotic behavior of wave propagation around a viscous shock profile}, Comm. Math. Phys., 290(1): 23--82,2009.


\bibitem{lz2015} T. Liu, Y. Zeng,  {\em Shock waves in conservation laws with physical viscosity},  Mem. Amer. Math. Soc., 234(1105): 2015.


\bibitem{mz2004} C. Mascia, K. Zumbrun, {\em Stability of large-amplitude viscous shock profiles of hyperbolic-parabolic system},  Arch. Ration. Mech. Anal., 172(1): 93--131,2004.



\bibitem{m} A. Matsumura, {\em Waves in compressible fluids: viscous shock, rarefaction, and contact waves},  Handbook of mathematical analysis in mechanics of viscous fluids, Springer, Cham, 2495--2548, 2018.


\bibitem{mm1999} A. Matsumura, M. Mei, {\em  Convergence to travelling fronts of solutions of the $p$-system with viscosity in the presence of a boundary}, Arch. Ration. Mech. Anal., 146(1), 1--22,1999.


\bibitem{mn1985} A. Matsumura, K. Nishihara, {\em On the stability of travelling wave solutions of a one-dimensional model system for compressible viscous gas},  Japan. J. Appl. Math., 2(1): 17--25,1985.


\bibitem{mn1994} A. Matsumura,  K. Nishihara, {\em Asymptotic stability of traveling waves for scalar viscous
conservation laws with non-convex nonlinearity}, Comm. Math. Phys., 165(1): 83--96,1994.


\bibitem{mn2004} A. Matsumura,  K. Nishihara, {\em  Global Solutions for Nonlinear Differential Equations-Mathematical
Analysis on Compressible Viscous Fluids (In Japanese)},  Nippon Hyoronsha, 2004.


\bibitem{mw2010} A. Matsumura, Y. Wang, {\em  Asymptotic stability of viscous shock wave for a one-dimensional isentropic model of viscous gas with density dependent viscosity},  Methods Appl. Anal., 17(3): 279--290,2010.


\bibitem{s1983} J. Smoller, {\em Shock Waves and Reaction-Diffusion Equations},  Springer-Verlag, New York, Berlin, 1983.


\bibitem{sx1993} A. Szepessy, Z. Xin,{\em  Nonlinear stability of viscous shock waves},  Arch. Ration. Mech. Anal., 122(1): 53--103,1993.


\bibitem{vy2016} A. Vasseur, L. Yao, {\em Nonlinear stability of viscous shock wave to one-dimensional compressible isentropic Navier-Stokes equations with density dependent viscous coefficient},  Commun. Math. Sci., 14(8): 2215--2228,2016.



\bibitem{z2004} K. Zumbrun, {\em  Stability of large-amplitude shock waves of compressible Navier-Stokes equations}, with an appendix by Helge Kristian Jenssen and Gregory Lyng    Handbook of Mathematical Fluid Dynamics, vol. III, North-Holland, Amsterdam, 311--533, 2004.

          \end{thebibliography}
          \end{document}